\documentclass[11pt]{article}
\usepackage[paperwidth=7.5in,paperheight=11in,top=1.25in, bottom=1.25in, left=1.00in, right=1.00in]{geometry}
\usepackage[english]{babel}
\usepackage{amsmath,amsfonts,amssymb}
\usepackage{color}
\usepackage[latin1]{inputenc}
\usepackage{fancyhdr}
\usepackage{hyperref}
\usepackage{enumerate}
\usepackage[textsize=footnotesize]{todonotes} 

\newtheorem{Theorem}{Theorem}

\newtheorem{Proposition}{Proposition}
\newtheorem{Lemma}{Lemma}
\newtheorem{Corollary}{Corollary}
\newtheorem{remark}{Remark}[part]

\allowdisplaybreaks

\newcommand \Sum{\displaystyle\sum}

\newcommand \Sup{\displaystyle\sup}

\renewcommand \P{\mathbb{P}}
\newcommand \R{\mathbb{R}}

\newcommand \Bc{\mathcal{B}}

\newcommand \Fc{\mathcal{F}}

\newcommand \ep{\hbox{ }\hfill$\Box$}

\newcommand \dd{{\rm d}}
\newcommand \dr{\dd r}
\newcommand \ds{\dd s}

\newcommand \dW{\dd W}
\newcommand \dB{\dd B}

\title{Zhang $L^2$-Regularity for the solutions of Backward Doubly Stochastic Differential Equations under globally Lipschitz continuous assumptions}

\def\first{\vtop{\baselineskip=11pt
\hbox to 125truept{\hss Achref Bachouch\footnote{Matematisk institutt, Universitetet i Oslo. Postboks 1053 Blindern,
0316 Oslo, Norway. Email: {\tt  achref.bachouch@gmail.com}.}
\hss } \vskip2truept
 \hbox to
125truept{\small\hss Universitetet i Oslo   \hss}
\hbox to
125truept{\small\hss   \hss}  }}

 \def\third
{\vtop{\baselineskip=11pt \hbox to 125truept{\hss Anis MATOUSSI
\footnote{Risk and Insurance Institute and Laboratoire Manceau de Mathématiques, University of Le Mans, Avenue Olivier Messiaen, 72085 Le Mans Cedex 09, France. Email: {\tt anis.matoussi@univ-lemans.fr }. The second author's research is part of Chair Financial Risks of Risk Foundation at Ecole Polytechnique and the ANR project CAESARS (ANR-15-CE05-0024).
} 
}
 \vskip2truept
 \hbox to 125truept{\small\hss University of Le Mans \hss}}}
\begin{document}
\author{\first \quad  \third}
\maketitle
\begin{abstract}
We prove an $L^2$-regularity result for the solutions of Forward Backward Doubly Stochastic Differentiel Equations (F-BDSDEs in short) under globally Lipschitz continuous assumptions on the coefficients. Therefore, we extend the well known regularity results established by Zhang (2004) for Forward Backward Stochastic Differential Equations (F-BSDEs in short) to the doubly stochastic framework. 
To this end, we prove (by Malliavin calculus) a representation result for the martingale component of the solution of the F-BDSDE under the assumption that the coefficients are continuous in time and continuously differentiable in space with bounded partial derivatives. As an (important) application of our $L^2$-regularity result, we derive the rate of convergence in time for the (Euler time discretization based) numerical scheme for F-BDSDEs proposed by Bachouch et al.(2016) under only globally Lipschitz continuous assumptions.

Keywords: Forward Backward Doubly Stochastic Differential Equations; $L^2$-regularity, Malliavin calculus; representation result; numerical scheme; rate of convergence

MSC2010: Primary 60H10, Secondary 65C30 
\end{abstract}

\section{Introduction}
\label{introduction:section}
Stochastic partial differential equations (SPDEs in short) appear in many applications, like Zakai equations in non linear filtering, stochastic control with partial observations and genetic populations. The SPDE of our interest is of the following form 
\begin{align}\label{SPDE}
u_t(x)=\Phi(x)&+\int_t^T [ L 
u_s(x) + f(s,x,u_s(x), ( \nabla u_s \sigma)(x)) ] \ds \nonumber \\
&+\int_t^T h(s,x,u_s(x), (\nabla u_s \sigma)(x)) \overleftarrow{ dB_{s} } ,
\end{align}
where, $T>0$ is fixed, $u_t ( x )=u ( t,x)$ is a predictable random field, $f$ and $h$ are non-linear deterministic  coefficients, $ L u= \big( Lu_1, \cdots,  Lu_k \big)$ 
is a second order differential operator and $\sigma$ is the diffusion coefficient. 
The differential term with $\overleftarrow{d B}_t$
refers to the backward stochastic integral with respect to an $l$-dimensional Brownian motion on
$\big(\Omega, \mathcal{F},P, (B_t)_{t\geq 0} \big)$.
F-BDSDEs have been  introduced to give a Feynman-Kac representation for the classical solution of the stochastic semilinear PDE \eqref{SPDE}, see the seminal work of \cite{pard:peng:94}. The BDSDE $(Y^{t,x},Z^{t,x})$ of our interest is of the following form
\begin{align}\label{BDSDE1}
Y_{s}^{t,x}&= \Phi(X_{T}^{t,x}) +\int_{s}^{T}f(r,X_{r}^{t,x},Y_{r}^{t,x},Z_{r}^{t,x})\dr\\
&+ \int_{s}^{T}h(r,X_{r}^{t,x},Y_{r}^{t,x},Z_{r}^{t,x})\overleftarrow{\dB_{r}}
- \int_{s}^{T}Z_{r}^{t,x}\dW_{r},\nonumber
\end{align}
where $(X_{s}^{t,x})_{ t\leq s\leq T}$ is a $d$-dimensional diffusion process  starting from $x$ at time $t$ driven by 
the finite $d$-dimensional Brownian motion $(W_t)_{0\leq t\leq T}$ (independent from $B$) with infinitesimal generator  $L$. Under  some regularity assumptions on the coefficients $b,\sigma,\Phi,f$ and $h$,
the authors in \cite{pard:peng:94}  proved that $u_t (x) =Y_{t}^{t,x}$ and $\nabla u_t  \sigma (x)  =Z_t^{t,x}$, $\forall (t,x)\in [0,T]\times\mathbb{R}^{d}$ (see \cite[Theorem 3.1]{pard:peng:94} for details).
Many generalizations studying more general nonlinear SPDEs have been made by different approaches of the notion of weak solutions, that is,  Sobolev's solutions ( see \cite{ kryl:99,ball:mato:01,mato:sche:02})  and  stochastic viscosity  solutions (see \cite{lion:soug:98,buck:ma:01, lion:soug:02}).\\
Essentially, SPDEs have been numercially resolved by an analytic approach, that is, based on time-space discretization  of the equations. The discretization is achieved by different methods such as finite difference, finite element and spectral Galerkin methods \cite{gyon:nual:95, gyon:99, wals:05, gyon:kryl:10, jent:kloe:10}. More precisely, the Euler finite-difference scheme was studied in \cite{gyon:nual:95}, \cite{gyon:99} and  \cite{gyon:kryl:10}. Its convergence was proved in \cite{gyon:nual:95} and the order of convergence was determined in \cite{gyon:99}. Very  interesting results are presented in \cite{gyon:kryl:10} when
they studied a  symmetric finite difference scheme for a class
  of linear SPDEs driven by an infinite dimensional Brownian motion. The authors proved that the approximation error is proportional to
   $\bar{h}^2$
 where $\bar{h}$ is the discretization step in space. They even proved (using the Richardson acceleration method) that if the SPDE is non degenerate
and the coefficients are m-times continuously differentiable in the state variable (of dimension $d$) with $m>1+\frac{d}{2}$, then the error is proportional to  $\bar{h}^4$.
Finite element based schemes for parabolic SPDEs were studied in \cite{wals:05} in the one-dimensional case, with a study of the rate of convergence for the Forward and Backward Euler and the Crank-Nicholson schemes. The obtained rate of convergence is similar to the rate of the finite difference schemes. The spectral Galerkin approximation was investigated in \cite{jent:kloe:10}. The method is based on Taylor expansions derived from the solution of the SPDE, under sufficient regularity conditions.
This approach was also used in \cite{loto:Miku:Rozo:97} to approximate the solution of the Zakai equation.\\
Only recently some works took an active interest in the simulation and approximation of \eqref{BDSDE1}. This interest was motivated by results and advances in the approximation and simulation of the standard F-BSDEs during the last fifteen years. Indeed, when $h\equiv 0$, SPDE \eqref{SPDE} becomes a deterministic PDE and we deal with a standard F-BSDE. The numerical resolution of F-BSDEs has already been
studied in the literature by Bally \cite{ball:97}, Zhang \cite{zhan:04}, Bouchard and Touzi \cite{bouc:touz:04}, Gobet, Lemor and Warin \cite{gobe:lemo:wari:06} and Bouchard and Elie
\cite{bouc:elie:08} among others. Zhang \cite{zhan:04} suggested a discrete-time approximation, by step processes,
for a class of decoupled F-BSDEs with possible path-dependent terminal values. He established an $L^2$-type regularity result for
the F-BSDE's solution. Then he proved the convergence of his numerical scheme and he derived the rate of convergence in time. Bouchard and Touzi \cite{bouc:touz:04} proposed a similar numerical scheme for decoupled F-BSDEs. They computed the conditional expectations involved in their numerical scheme using the kernel regression estimation and used the Malliavin approach and the Monte carlo method for their computation. Gobet, Lemor and Warin in \cite{gobe:lemo:wari:06} suggested an explicit (time discretization based) numerical scheme. They also proposed an empirical regression scheme to approximate the nested conditional expectations arising from the time discretization of the standard F-BSDE. The latter method,  also known as regression Monte-Carlo method or least-squares Monte-Carlo method, is popular and known to perform well
for high-dimensional problems.\\
In the stochastic PDEs' case, that is $h\neq 0$, Aman \cite{aman:13} and Aboura \cite{abou:11} considered the particular case when
  $h$ does not depend on the control variable $z$. Aman \cite{aman:13} proposed a numerical scheme following Bouchard and Touzi  \cite{bouc:touz:04} and obtained a convergence of order $\hat{h}$ of the square of the $L^2$- error ($\hat{h}$ is the time discretization step).  Aboura \cite{abou:11} studied the same numerical scheme under the same kind of assumptions, but following Gobet et al. \cite{gobe:lemo:wari:05}. He obtained a convergence of order $\hat{h}$  in time and used the regression Monte Carlo method to implement his scheme, as in  \cite{gobe:lemo:wari:05}. Also, when $h$ doesn't depend on $z$, a first order scheme was proposed in \cite{bao:cao:meir:zhao:16} using the two sided
Ito-Taylor expansion when the forward process is a drifted Brownian motion. Under the assumption that the coefficients are 3 times continuously differentiable with all partial derivatives bounded, they obtained a rate of convergence of order $\hat{h}^2$ for the component   $y$ and of order $\hat{h}$ for the component $z$.  \\
In the general case, that is, $h$ depends on the variable $z$, 
 the authors in \cite{bao:cao:zhao:11} studied the time discretization error for the time discretization based approximation scheme for F-BDSDEs when the forward process is simply a drifted Brownian motion. They derived a rate of convergence of order $\hat{h}$ under the assumption that the coefficients are continuously differentiable in space with all partial derivatives uniformly bounded.
In \cite{bach:benl:mato:mnif:16}, the authors extended the approach of Bouchard-Touzi-Zhang to F-BDSDEs. They gave an upper bound for the time discretization error under globally Lipschitz continuous assumptions on the coefficients. However, they derived the rate of convergence of this scheme under rather strong assumptions, namely, all the coefficients are 2 times continuously differentiable with all partial derivatives bounded. Finally, they deduced a numerical scheme for the weak solution of the  semilinear SPDE \eqref{SPDE} and gave the rate of convergence in time for the latter numerical scheme. The problem of approximation of nested conditional expextations arising from the time discretization of F-BDSDEs was recently resolved in \cite{bach:gobe:mato:16} using the regression Monte-Carlo method. The resolution was done conditionally to the paths of the Brownian motion $B$, in the spirit of SPDE \eqref{SPDE}, under globally Lipschitz continuous assumptions on the coefficients.\\
This leads to the motivation of this paper. For the numerical resolution of F-BDSDEs with coefficient $h$ depending on $z$, the regression Monte-Carlo scheme studied in \cite{bach:gobe:mato:16} converges under Lipschitz continous assumptions on the coefficients, while the rate of convergence in time is proved in \cite{bach:benl:mato:mnif:16} under the assumption that the coefficients are 2 times continuously differentiable with all partial derivatives bounded. A natural problem of interest is to derive the same rate of convergence in time (obtained in \cite{bach:benl:mato:mnif:16}) under only Lipschitz continuous assumptions on the coefficients. This enables the approximation and simulation of F-BDSDEs under only Lipschitz continuous conditions on the coefficients, which in turn enables the numerical approximation of weak solutions of SPDE \eqref{SPDE} (via F-BDSDEs) under rather mild conditions. \\
To this end, we proceed as follows. First, we study the Malliavin derivative of the solution $(Y,Z)$ of the F-BDSDE and we prove a representation result for the martingale component $Z$ of the solution under the assumption that all the coefficients are continuous in time and continuously differentiable in space with all partial derivatives uniformly bounded. Afterwards, we use this representation result to prove an $L^2$-regularity result for the solution of the F-BDSDE under globally Lipschitz continuous assumptions, extending the well known $L^2-$ regularity results for F-BSDEs (proved by Zhang in \cite{zhan:04}) to the doubly stochastic framework. Then, our $L^2-$ regularity result is used to derive the rate of convergence in time of the numerical scheme for F-BDSDEs proposed in \cite{bach:benl:mato:mnif:16}  under globally Lipschitz continuous assumptions.  \\ 
 
The paper is organized as follows. In section 2, we make a recall about F-BDSDES and introduce the notations and the assumptions needed in our work. In section 3, we prove a representation result for the martingale component $Z$ of the solution. Then, we prove our main result which is an $L^2$-regularity result for the solution of the F-BDSDE. In section 4, we apply our $L^2$-regularity result to derive  the rate of convergence in time of the numerical scheme for F-BDSDEs proposed in \cite{bach:benl:mato:mnif:16}.

\paragraph{Usual notations.}
{If $x$ is in an Euclidean space $E$, $|x|$ denotes its norm. If $A$ is a matrix, $|A|$ stands for its Hilbert-Schmidt norm.}

\section{Notations, preliminaries on Forward Backward Doubly Stochastic Differential Equations and assumptions}
\label{FBDSDE:subsection}

We assume that $W$ and $B$ are two independent Brownian motions defined on a filtered probability space $\big(\Omega, \Fc,\P\big)$, where we define the sigma-fields $\Fc_{t,s}^{W}:= \sigma\{W_{r}-W_{t}, t\leq r\leq s\}$, $\Fc_{s,T}^{B}:= \sigma\{B_{r}-B_{s}, s\leq r\leq T \}$, $\Fc^W:=\Fc_{0,T}^{W}$, $\Fc^B:=\Fc_{0,T}^{B}$, 
  $\Fc:= \Fc^{W}\vee \Fc^{B}$, all completed with the $\P$-null sets. The solution of \eqref{BDSDE1} is measurable at each $s$ with respect to  $\Fc_{s}^{t}$, where $(\Fc_{s}^{t})_{t \leq s \leq T}$ is the collection of sigma-fields defined as follows. For fixed $t\in[0,T]$ and for all $s \in [t,T]$
\begin{equation}
\Fc_{s}^{t}:= \Fc_{t, s}^{W}\vee \Fc_{s,T}^{B}\label{eq:fst}. \nonumber
\end{equation}
We denote $\mathcal{F}_s^0$ by $\mathcal{F}_s$ for simplicity.\\

We also need to introduce the following spaces:\\
$\bullet$ $C_{b}^{l}(\mathbb{R}^{p},\mathbb{R}^{q})$ denotes the set
of all functions $\phi:\mathbb{R}^{p} \longrightarrow \R^{q}$ such that they are $l$-times continuously differentiable with all partial derivatives uniformly bounded. We denote  $C_{b}^{l}$ when the context is clear.\\
$\bullet$ $C_{b}^{k,l}([0,T] \times \mathbb{R}^{p},\mathbb{R}^{q})$ denotes the set
of all functions $\phi:[0,T] \times\mathbb{R}^{p} \longrightarrow \R^{q}$ such that they are $k$-times continuously differentiable in time and $l$-times continuously differentiable in space with all partial derivatives uniformly bounded. We denote  $C_{b}^{k,l}$ when the context is clear.\\
$\bullet$ $L^2(\Omega,\mathcal{F}_T,P;\mathbb{R}^k)$ denotes the set of $\mathcal{F}_T$-measurable square integrable
random variables with values in $\mathbb{R}^k$.\\
For any $m\in\mathbb{N}$ and $t\in [0,T] $,  the following notations are introduced:\\
$\bullet$ $\mathbb{H}^{2}_m([t,T])$
denotes the set of (classes of $dP\times dt$ a.e. equal) $\mathbb{R}^{m}$-valued jointly measurable processes $\{\psi_{u}; u\in [t,T]\}$ satisfying:\\
(i) $||\psi||^2_{\mathbb{H}^{2}_m([t, T])}:=E[\int_{t}^{T}|\psi_{u}|^{2}du]<\infty$,\\
(ii) $\psi_{u} $ is $ \mathcal{F}_{u}$-measurable, for a.e. $u\in [t,T]$.\\
$\bullet$ $\mathbb{S}^{2}_m([t, T])$ denotes similarly the set of $\mathbb{R}^{m}$-valued continuous processes satisfying:\\
(i) $||\psi||^2_{\mathbb{S}^{2}_m([t, T])}:=E[\sup_{t\leq u\leq T}|\psi_{u}|^{2}]<\infty$,\\
(ii) $\psi_{u} $ is $ \mathcal{F}_{u}$-measurable, for any $u\in [t,T]$.\\
$\bullet$  $\mathbb{S}$ denotes the set of random variables $F$ of the form $$F=\hat f(W(h_{1}),\ldots,W(h_{m_1}),B(k_{1}),\ldots,B(k_{m_{2}})),$$ 
where $\hat f\in C_{b}^{\infty}(\mathbb{R}^{m_1+m_2},\R)$,
$h_{1},\ldots,h_{m_1}\in L^{2}([t,T],\mathbb{R}^{d}),k_{1},\ldots,k_{m_2}\in L^{2}([t,T],\mathbb{R}^{l})$,
\begin{align*}
W(h_{i}):=\int_{t}^{T}h_{i}(s)dW_{s} \textrm{ and }
B(k_{j}):=\int_{t}^{T}k_{j}(s)\overleftarrow{dB_{s}}.
\end{align*}
For any random variable $F \in \mathbb{S}$,  its Malliavin derivative $(D_{s}F)_s$ is defined  with respect to the Brownian motion
 $W$ as follows
\begin{eqnarray*}
 D_{s}F := \sum_{i=1}^{m_1}\nabla_{i} \hat f \bigg(W(h_{1}),\ldots,W(h_{m_1});B(k_{1}),\ldots,B(k_{m_{2}})\bigg)h_{i}(s),
\end{eqnarray*}
where $\nabla_{i}\hat f$ is the derivative of $\hat f$ with respect to its i-th argument.\\
We define a norm on $\mathbb{S}$  by:
\begin{eqnarray*}
\|F\|_{1,2}:=\big\{E[|F|^{2}]+\int_{t}^{T}E[|D_{s}F|^{2}]ds \big\}^{\frac{1}{2}}.
\end{eqnarray*}
$\bullet$ $\mathbb{D}^{1,2}:= \overline{{\mathbb S}}^{\|.\|_{1,2}}$ is then a Sobolev space.\\
$\bullet$ $\mathcal{S}_k^2([t,T],\mathbb{D}^{1,2})$ is the set of processes $Y=(Y_u,t\leq u\leq T)$
such that $Y\in \mathbb{S}_k^2([t,T])$, $Y^i_u\in \mathbb{D}^{1,2}$, $1\leq i\leq k$, $t\leq u\leq T$ and
$$\|Y\|_{1,2}:=\Big\{E\Big[\int_{t}^{T}|Y_u|^2du + \int_{t}^{T}\int_{t}^{T}|D_{\theta}Y_u|^2du d\theta \Big] \Big\}^{\frac{1}{2}}<\infty.$$
$\bullet$ $\mathcal{M}_{k\times d}^2([t,T],\mathbb{D}^{1,2})$ is the set of processes $Z=(Z_u,t\leq u\leq T)$
such that $Z\in \mathbb{H}_{k\times d}^2([t,T])$, $Z^{i,j}_u\in \mathbb{D}^{1,2}$,$1\leq i\leq k$, $1\leq j\leq d$, $t\leq u\leq T$ and
$$ \|Z\|_{1,2}:=\Big\{E\Big[\int_{t}^{T}|Z_u |^2du+\int_{t}^{T}\int_{t}^{T}|D_{\theta}Z_u|^2du d\theta \Big] \Big\}^{\frac{1}{2}}<\infty.$$
$\bullet$ $\Bc^2([t,T],\mathbb{D}^{1,2}):={\cal S}_k^2([t,T],\mathbb{D}^{1,2})\times {\cal M}_{k\times d}^2([t,T],\mathbb{D}^{1,2})$.\\
We define also for a given $t \in [0,T]$:\\
$\bullet$  $L^2([t,T],\mathbb{D}^{1,2})$ is the set of processes $(v_s)_{t\leq s\leq T}$ such that  $v_s$ is measurable with respect to $\mathcal{F}_s^t$ for a.e. $s$ and \\
(i) $v(s,.)\in \mathbb{D}^{1,2}$, for a.e. $s\in [t,T]$,\\
(ii) $(s,w)\longrightarrow Dv(s,w)\in L^2([t,T]\times \Omega)$,\\
(iii) $E[\int_t^T|v_s|^2ds]+E[\int_t^T\int_t^T|D_uv_s|^2duds]<\infty$.\\
$\bullet$ $L^2([t,T],\mathbb{D}^{1,2}\times \mathbb{D}^{1,2}):=L^2([t,T],\mathbb{D}^{1,2})\times L^2([t,T],\mathbb{D}^{1,2})$.\\

For all $(t,x) \in [0,T]\times\mathbb{R}^{d}$, let $(X_{s}^{t,x})_{t \leq s \leq T}$ be the unique strong solution of the following  stochastic differential equation:
\begin{eqnarray}\label{forward}
dX_{s}^{t,x}=b(X_{s}^{t,x})ds+\sigma(X_{s}^{t,x})dW_{s},\quad s\in [t,T],\qquad X_{s}^{t,x}=x,\quad 0\leq s\leq t,
\end{eqnarray}
where $b$ and $\sigma$ are two functions on $\mathbb{R}^{d}$ with values respectively in $\mathbb{R}^{d}$ and $\mathbb{R}^{d\times d}$. \\
We consider the following BDSDE: For all $ t\leq s\leq T$,
\begin{equation}\label{1}
\left\{
\begin{array}{ll}
dY_{s}^{t,x}&= -f(s,X_{s}^{t,x},Y_{s}^{t,x},Z_{s}^{t,x})ds -h(s,X_{s}^{t,x},Y_{s}^{t,x},Z_{s}^{t,x})\overleftarrow{dB_{s}}
+Z_{s}^{t,x} dW_s,\\
Y_{T}^{t,x}&=\Phi(X_{T}^{t,x}),
\end{array}
\right.
\end{equation}
where $f$ and $\Phi$ are two functions respectively on $[0,T]\times\mathbb{R}^{d}\times\mathbb{R}^{k}\times\mathbb{R}^{k\times d}$ and $\mathbb{R}^{d}$ with values in $\mathbb{R}^{k}$ and
 $h$ is a function on $[0,T]\times \mathbb{R}^{d}\times\mathbb{R}^{k}\times\mathbb{R}^{k\times d}$ with values in $\mathbb{R}^{k\times l}$.\\
 We will omit the dependence of the process $X$ on the initial condition if it starts at time $t=0$. \\ 
The following assumptions will be needed in our work.\\
\textbf{Assumption (H1)} There exists a non-negative constant $K$ such that
\begin{eqnarray*}
 &&|b(x)-b(x')| + |\sigma(x) - \sigma(x')| \leq K|x-x'|,
\forall x, x' \in \mathbb{R}^{d}.
\end{eqnarray*}
\textbf{Assumption (H2)} There exist two constants $K\geq 0$ and $0\leq \alpha < 1$ such that\\
for any $(t_1,x_{1},y_{1},z_{1}),(t_2,x_{2},y_{2},z_{2})\in [0,T]\times\mathbb{R}^{d}
\times\mathbb{R}^{k}\times \mathbb{R}^{k\times d} ,$
\begin{align*}
\textbf{(i)}&|f(t_1,x_{1},y_{1},z_{1})-f(t_2,x_{2},y_{2},z_{2})|
\leq K \big(\sqrt{|t_{1}-t_{2}|}+|x_{1}-x_{2}|+|y_{1}-y_{2}|\\
&\qquad \qquad \qquad\qquad\qquad \qquad\qquad\qquad+|z_{1}-z_{2} |\big),\\
\textbf{(ii)}&|  h(t_1,x_{1},y_{1},z_{1})  -  h(t_2, x_{2},  y_{2},  z_{2})|^{2}
\leq
 K  \big(|t_{1}-t_{2}|\! +\! |x_{1} - x_{2}|^{2} + |y_{1} - y_{2}|^{2} \big)\\
&\qquad \qquad \qquad\qquad\qquad \qquad\qquad\qquad +\! \alpha^{2} | z_{1}\! -\! z_{2} |^{2},\\
\textbf{(iii)}&|\Phi(x_1)-\Phi(x_2)|\leq K|x_{1}-x_{2}|,\\
\textbf{(iv)}&\Sup_{0\leq t\leq T}(|f(t,0,0,0)|+|h(t,0,0,0)|)\leq K .
\end{align*}
\begin{remark}
Pardoux and Peng \cite[Theorem 1.1]{pard:peng:94} proved that under assumptions {\bf{(H1)}} and {\bf{(H2)}}, there exists a unique solution $(Y,Z)$ in $ \mathbb{S}^{2}_k([t,T])\times \mathbb{H}^{2}_{k\times d}([t,T])$ to the F-BDSDE \eqref{forward}-\eqref{1}.
\end{remark}
From \cite{elka:peng:quen:97}, \cite{pard:peng:94} and \cite{kuni:84}, the following standard estimates for the solution of the F-BDSDE \eqref{forward}-\eqref{1} hold
and we remind the following theorem.
\begin{Theorem}
Under assumptions \textbf{(H1)} and \textbf{(H2)}, there exists a positive constant $C$ such that
\begin{eqnarray}\label{integrability}
E[\Sup_{t\leq s\leq T}|X_{s}^{t,x}|^{2}]\leq C(1+|x|^2),
\end{eqnarray}
\begin{equation}\label{apriori1}
E\Big[\Sup_{t\leq s\leq T}|Y_{s}^{t,x}|^{2}+ \int_{t}^{T}|Z_{s}^{t,x}|^{2}ds  \Big]
\leq C(1+|x|^{2}).
\end{equation}
\end{Theorem}

\section{Representation result and Zhang $L^2$-regularity}\label{section Zhang regularity}
The aim of this section is to prove an $L^2$-regularity result for the solution of the F-BDSDE \eqref{forward}-\eqref{1} under globally Lipschitz continuous assumptions on the coefficients. To this end, we prove
a representation and a path regularity results for the martingale component $Z$ of the solution under the assumption that the coefficients $b,\sigma$ and $\Phi$ are in $C_b^{1}$ and $f$ and $h$ are  in $C_b^{0,1}$. These results enable us to derive a rate of convergence in time for the numerical scheme for F-BDSDEs studied in \cite{bach:benl:mato:mnif:16} (see subsections \ref{Num scheme BDSDEs} and \ref{section-discrete-time-error} for details) under only globally Lipschitz continous assumptions on the coefficients $b,\sigma,f,h$ and $\Phi$, in the spirit of the results of Zhang \cite{zhan:04}.\\
Let us stress that the representation  and the path regularity results for the component $Z$  are proved in \cite{bach:benl:mato:mnif:16} under the assumption that all the coefficients are 2 times continuously differentiable with all partial derivatives uniformly bounded. Here, the proofs are given under weaker assumptions compared to \cite{bach:benl:mato:mnif:16}.
\subsection{Malliavin calculus for the solutions of forward SDEs}
In this subsection, we recall some results on the differentiability in the Malliavin sense of the forward process $X^{t,x}$. Under the assumption that $b$ and $\sigma$ are in $C_b^{1}$, Nualart \cite{nual:06}  stated that $X_{s}^{t,x} \in \mathbb{D}^{1,2}$ for any $s\in[t,T]$ and for $l\leq k$,
the derivative $D_{r}^{l}X_{s}^{t,x}$ is given by:\\
(i) $D_{r}^{l}X_{s}^{t,x} = 0,\quad for\quad s<r\leq T$,\\
(ii) For any\quad$t<r\leq T$, a version of $\{D_{r}^{l}X_{s}^{t,x}, r\leq s\leq T \}$ is the unique solution of the following linear SDE
\begin{equation}\label{Mall Deriv X SDE}
D_{r}^{l}X_{s}^{t,x} = \sigma^{l}(X_{r}^{t,x}) + \int_{r}^{s}\nabla b(X_{u}^{t,x})D_{r}^{l}X_{u}^{t,x}du
+\sum_{i=1}^{d}\int_{r}^{s}\nabla\sigma^{i}(X_{u}^{t,x})D_{r}^{l}X_{u}^{t,x}dW_{u}^{i},
\end{equation}
where $(\sigma^{i})_{i=1,\ldots,d}$ denotes the i-th column of the matrix $\sigma$.\\
The following inequalities will be useful later. From \cite{nual:06}, we know that for any $0 \leq r \leq s \leq T$, there exists a non-negative constant $C$ such that
\begin{equation}\label{aprioriDX}
E\Big[\Sup_{0\leq u\leq T} |D_s X_u|^2\Big] \leq C (1+|x|^{2}),
\end{equation}
\begin{equation}\label{aprioriDX2}
E[\Sup_{s \vee r \leq u \leq T} |D_{s}X_{u}-D_{r}X_{u}|^{2}]\leq C |s-r|(1+|x|^{2}).
\end{equation}

\subsection{Malliavin calculus for the solutions of F-BDSDEs}
In this subsection, we study the differentiability in the Malliavin sense of the solution of the F-BDSDE \eqref{forward}-\eqref{1}. First, we recall the following result from Pardoux and Peng \cite{pard:peng:92} about the Malliavin derivative of the classical It\^o integral.
\begin{Lemma}[\cite{pard:peng:92}]\label{forD12}
Let $U\in \mathbb{H}_1^2([t,T])$ and $I_{i}(U)=\int_{t}^{T}U_{r} dW^{i}_{r},
i=1,\ldots,d$.  Then, for each $\theta\in[0,T]$ we have $U_{\theta}\in \mathbb{D}^{1,2}$
if and only if $I_{i}(U)\in\mathbb{D}^{1,2}, i=1,\ldots,d$ and for all $\theta \in [0,T]$, we have
\begin{eqnarray*}
D_{\theta}I_{i}(U)&=&\int_{\theta}^{T}D_{\theta}U_{r}dW_{r}^{i}+ U_{\theta},\,\,\theta>t,\\
D_{\theta}I_{i}(U)&=&\int_{t}^{T}D_{\theta}U_{r}dW_{r}^{i},\,\,\theta\leq t.
\end{eqnarray*}
\end{Lemma}
We also recall the following lemma from \cite{bach:benl:mato:mnif:16} which shows that a backward It\^o integral is differentiable in the
Malliavin sense if and only if its integrand is so. More precisely, since the Malliavin derivative is with respect to the Brownian motion W, we have
\begin{Lemma}[\cite{bach:benl:mato:mnif:16}]\label{backD12}
Let $U\in \mathbb{H}_1^2([t,T])$ and $I_{i}(U)=\int_{t}^{T}U_{r} \overleftarrow{dB^{i}_{r}},
i=1,\ldots,l$. Then for each $\theta\in[0,T]$ we have $U_{\theta}\in \mathbb{D}^{1,2}$ if and only if
$I_{i}(U)\in\mathbb{D}^{1,2}, i=1,\ldots,l$ and for all $\theta \in [0,T]$, we have
\begin{eqnarray*}
D_{\theta}I_{i}(U)&=&\int_{\theta}^{T}D_{\theta}U_{r}\overleftarrow{dB^{i}_{r}},\,\,\theta>t,\\
D_{\theta}I_{i}(U)&=&\int_{t}^{T}D_{\theta}U_{r}\overleftarrow{dB^{i}_{r}},\,\,\theta\leq t.
\end{eqnarray*}
\end{Lemma}
The following result will be needed to prove Proposition \ref{propderiv1}. It can be proved using the same arguments as in the classical BSDEs' setting (see \cite{elka:peng:quen:97}).
\begin{Proposition}\label{aprioriest}
Let  $(\phi^1,f^1,h^1)$ and
$(\phi^2,f^2,h^2)$ be two standard parameters of the BDSDE \eqref{1} and
$(Y^1,Z^1)$ and $(Y^2,Z^2)$ the associated solutions. Let assumptions {\bf (H1)} and {\bf (H2)} hold. For $s\in [t,T]$,
set $\delta Y_s:=Y^1_s-Y^2_s$,
$\delta_2f_s:=f^1(s,X_s,Y^2_s,Z^2_s)-f^2(s,X_s,Y^2_s,Z^2_s)$ and $\delta_2h_s:=h^1(s,X_s,Y^2_s,Z^2_s)-h^2(s,X_s,Y^2_s,Z^2_s)$.
Then, we have
\begin{eqnarray}\label{apriori2}
||\delta Y||_{\mathbb S^2_d([t,T])}^2+||\delta Z||_{\mathbb H^2_{d\times k}([t,T])}^2
\leq C E[|\delta Y_T|^2+\int_t^T|\delta_2f_s|^2ds+\int_t^T |\delta_2h_s|^2ds],
\end{eqnarray}
where $C$ is a positive constant depending only on $K$, $T$ and $\alpha$.
\end{Proposition}
In the next proposition, we prove that the Malliavin derivative of the solution of the BDSDE \eqref{1} is a solution of a linear BDSDE (see \cite{pard:peng:92} for the standard BSDEs' case). The same proposition is proved in \cite{bach:benl:mato:mnif:16} under the assumption that all the coefficients are 2 times continuously differentiable with all partial derivatives uniformly bounded. We give here the proof under weaker assumptions.

\begin{Proposition}\label{propderiv1}
Assume that {\bf(H1)} and {\bf(H2)} hold and that the coefficients $b,\sigma$ and $\Phi$ are in $C_b^{1}$ and $f$ and $h$ are  in $C_b^{0,1}$.
For any $t\in[0,T]$ and $x\in \mathbb{R}^{d}$, let $\{(Y_{s},Z_{s}),t\leq s\leq T\}$
denote the unique solution of the following BDSDE
\begin{eqnarray*}
Y_{s}&=& \Phi(X_{T}^{t,x}) +\int_{s}^{T}f(r,X_{r}^{t,x},Y_{r},Z_{r})dr+ \int_{s}^{T}h(r,X_{r}^{t,x},Y_{r},Z_{r})
\overleftarrow{dB_{r}}\nonumber\\
& -& \int_{s}^{T}Z_{r}dW_{r},\,\,t\leq s\leq T.
\end{eqnarray*}
Then,
 $(Y, Z) \in \Bc^2([t,T],\mathbb{D}^{1,2})$ and $\{D_{\theta}Y_{s},D_{\theta}Z_{s};t\leq s,\theta\leq T\}$ is given by\\
(i) $D_{\theta}Y_{s}=0, D_{\theta}Z_{s}=0$ for all  $t\leq s<\theta\leq T$,\\
(ii) for any fixed $\theta\in [t,T]$, $\theta\leq s\leq T$ and $1\leq i\leq d$, a version of $(D_{\theta}^{i}Y_{s},D_{\theta}^{i}Z_{s})$ is the unique
solution of the following BDSDE
\begin{eqnarray}\label{DY}
D_{\theta}^{i}Y_{s}&=&\nabla\Phi(X_{T}^{t,x})D_{\theta}^{i}X_{T}^{t,x}+\int_{s}^{T}\Big( \nabla_{x}f(r,X_{r}^{t,x},Y_{r},Z_{r})D_{\theta}^{i}X_{r}^{t,x}\Big)dr\nonumber\\
&+&\int_{s}^{T}\Big( \nabla_{y}f(r,X_{r}^{t,x},Y_{r},Z_{r})D_{\theta}^{i}Y_{r}
+\Sum_{j=1}^{d}\nabla_{z^j}f(r,X_{r}^{t,x},Y_{r},Z_{r})
 D_{\theta}^{i}Z_{r}^j\Big)dr\nonumber\\
&+&\Sum_{n=1}^{l}\int_{s}^{T}\Big( \nabla_{x}h^n(r,X_{r}^{t,x},Y_{r},Z_{r})D_{\theta}^{i}X_{r}^{t,x}
+\nabla_{y}h^n(r,X_{r}^{t,x},Y_{r},Z_{r})D_{\theta}^{i}Y_{r}\Big)\overleftarrow{dB_{r}^{n}}\nonumber\\
&+&\Sum_{n=1}^{l}\int_{s}^{T}\Sum_{j=1}^{d}\Big( \nabla_{z^j}h^n(r,X_{r}^{t,x},Y_{r},Z_{r})
D_{\theta}^{i}Z_{r}^j\Big)\overleftarrow{dB_{r}^n}- \int_{s}^{T}\Sum_{j=1}^{d}D_{\theta}^{i}Z_{r}^jdW_{r}^j,
\end{eqnarray}
where $(z^j)_{1\leq j\leq d}$ denotes the j-th column of the matrix $z$,
$(h^n)_{1\leq n\leq l}$ denotes the n-th column of the matrix h and $B=(B^1,\ldots,B^l)$.
\end{Proposition}
\begin{proof}
To simplify the notations, we restrict ourselves to the case $k=d=l=1$.
$(D_{\theta}Y,D_{\theta}Z)$ is well defined and from inequalities \ref{apriori1} and \ref{aprioriDX}, we deduce that for each $\theta\leq T$
\begin{equation*}
E[\Sup_{t\leq s\leq T}|D_{\theta}Y_s|^2]+E[\int_t^T|D_{\theta}Z_s|^2ds]\leq C (1+|x|^{2}).
\end{equation*}
We define recursively the sequence $(Y^{m},Z^{m})$ as follows. First we set $(Y^{0},Z^{0})=(0,0)$. Then, given $(Y^{m-1},Z^{m-1})$,
we define $(Y^{m},Z^{m})$ as the unique solution in $\mathbb{S}_k^{2}([t,T])\times \mathbb{H}^2_{k\times d}([t,T])$ of
\begin{eqnarray*}
Y_{s}^{m}&=&\Phi(X_{T}^{t,x})+\int_{s}^{T}f(r,X_{r}^{t,x},Y_{r}^{m-1},Z_{r}^{m-1})dr
+\int_{s}^{T}h(r,X_{r}^{t,x},Y_{r}^{m-1},Z_{r}^{m-1})\overleftarrow{dB_{r}}\nonumber\\
&-& \int_{s}^{T}Z_{r}^{m}dW_{r}.
\end{eqnarray*}
We recursively show that $\displaystyle{(Y^{m},Z^{m})\in\mathcal{B}^{2}([t,T],\mathbb{D}^{1,2})}$.
Suppose that 
\begin{align*}
\displaystyle{(Y^{m},Z^{m})\in \Bc^2([t,T],\mathbb{D}^{1,2})}
\end{align*} 
and let us show that
\begin{align*}
(Y^{m+1},Z^{m+1})\in \Bc^2([t,T],\mathbb{D}^{1,2}).
\end{align*}
Set $\displaystyle{\Sigma_{r}^{m}:=(X_{r}^{t,x},Y_{r}^{m},Z_{r}^{m})}$. From the induction assumption, we have 
\begin{align*}
\Phi(X_{T})+\int_{s}^{T}f(r,\Sigma_{r}^{m})dr\in\mathbb{D}^{1,2}.
\end{align*}
We have $\displaystyle{h(r,\Sigma_{r}^{m})\in  \mathbb{D}^{1,2}}$ for all $r\in [t,T]$. From Lemma \ref{backD12},
we have 
\begin{align*}
\int_{t}^{T}h(r,\Sigma_{r}^{m})\overleftarrow{dB_{r}}\in \mathbb{D}^{1,2}.
\end{align*}
Then
\begin{eqnarray*}
Y_{s}^{m+1}=E\big[ \Phi(X_{T}^{t,x})+\int_{s}^{T}f(r,\Sigma_{r}^{m})dr+\int_{s}^{T}h(r,\Sigma_{r}^{m})\overleftarrow{dB_{r}}|\mathcal{F}_{t,s}^{W}\vee \mathcal{F}_{t,T}^{B}\big]\in\mathbb{D}^{1,2}
\end{eqnarray*}
Hence
\begin{eqnarray*}
\int_{t}^{T}Z_{r}^{m+1}dW_{r}=\Phi(X_{T}^{t,x})+\int_{t}^{T}f(r,\Sigma_{r}^{m})dr+
\int_{t}^{T}h(r,\Sigma_{r}^{m})\overleftarrow{dB_{r}}-Y_{t}^{m+1}\in\mathbb{D}^{1,2}.
\end{eqnarray*}
It follows from Lemma \ref{forD12} that $\displaystyle{Z^{m+1}\in \mathcal{M}^{2}_{k\times d}([t,T],\mathbb{D}^{1,2})}$
and we have for  $t\leq s\leq \theta$, $\displaystyle{D_{\theta}Y_{s}^{m+1}= D_{\theta}Z_{s}^{m+1}=0}$, while for $\theta\leq s \leq T$, we have
\begin{eqnarray}\label{eqiter}
&&D_{\theta}Y_{s}^{m+1}=\nabla\Phi(X_{T}^{t,x})D_{\theta}X_{T}^{t,x}\nonumber\\
&&+\int_{s}^{T}\Big( \nabla_{x}f(r,\Sigma_{r}^{m})
 D_{\theta}X_{r}+\nabla_{y}f(r,\Sigma_{r}^{m})D_{\theta}Y_{r}^{m}+\nabla_{z}f(r,\Sigma_{r}^{m})D_{\theta}Z_{r}^{m}\Big)dr\nonumber\\
 &&+\int_{s}^{T}\Big( \nabla_{x}h(r,\Sigma_{r}^{m})
 D_{\theta}X_{r}+ \nabla_{y}h(r,\Sigma_{r}^{m})D_{\theta}Y_{r}^{m}+\nabla_{z}h(r,\Sigma_{r}^{m})
 D_{\theta}Z_{r}^{m}\Big) \overleftarrow{dB_{r}} \nonumber\\
 &&- \int_{s}^{T}D_{\theta}Z_{r}^{m+1}dW_{r}.
\end{eqnarray}
From inequality \ref{apriori1}, we deduce that for each $\theta\leq T$
\begin{equation*}\label{aprioriDYm}
E[\Sup_{t\leq s\leq T}|D_{\theta}Y_s^{m+1}|^2]+\int_t^TE[|D_{\theta}Z_s^{m+1}|^2]ds\leq C (1+|x|^{2}).
\end{equation*}
It is known that inequality \ref{apriori1} holds for $(Y^{m+1},Z^{m+1})$ and we deduce that
\begin{equation*}
\|Y^{m+1}\|_{1,2}+\|Z^{m+1}\|_{1,2}<\infty,
\end{equation*}
which shows that $(Y^{m+1},Z^{m+1})\in \Bc^2([t,T],\mathbb{D}^{1,2})$.
Using the contraction mapping argument as in \cite{elka:peng:quen:97}, we deduce that  $(Y^{m+1},Z^{m+1})$ converges to
$(Y,Z)$ in $\mathbb{S}^2([t,T])\times \mathbb{H}^2([t,T])$.
We will show that $(D_{\theta}Y^{m},D_{\theta}Z^{m})$ converges to $(Y^{\theta},Z^{\theta})$
 in $L^2(\Omega \times [t,T]\times [t,T] ,dP\otimes dt\otimes dt)$, where $Y_{s}^{\theta}=Z_{s}^{\theta}=0$ for all $t\leq s\leq \theta $
and $(Y_{s}^{\theta},Z_{s}^{\theta},\theta\leq s\leq T)$ is the solution
 of the following BDSDE
\begin{eqnarray}\label{eqlim}
Y_{s}^{\theta}&=&\nabla\Phi(X_{T}^{t,x})D_{\theta}X_{T}^{t,x}\nonumber\\
&+&\int_{s}^{T}\Big(\nabla_{x}f(r,\Sigma_{r})
 D_{\theta}X_{r}+\nabla_{y}f(r,\Sigma_{r})Y_{r}^{\theta}+\nabla_{z}f(r,\Sigma_{r})
 Z_{r}^{\theta}\Big)dr\nonumber\\
 &+&\int_{s}^{T}\Big( \nabla_{x}h(r,\Sigma_{r})
 D_{\theta}X_{r}+\nabla_{y}h(r,\Sigma_{r})Y_{r}^{\theta}+\nabla_{z}h(r,\Sigma_{r})
 Z_{r}^{\theta}\Big)\overleftarrow{dB_{r}} \nonumber\\
 &-& \int_{s}^{T}Z_{r}^{\theta}dW_{r}.
\end{eqnarray}
From equations \ref{eqiter} and \ref{eqlim}, we have
\begin{align*}
&D_{\theta}Y_{s}^{m+1}-Y_{s}^{\theta}=
\int_{s}^{T}\Big((\nabla_{x}f(r,\Sigma_{r}^{m})
-\nabla_{x}f(r,\Sigma_{r}))D_{\theta}X_{r}^{t,x}\\
&+\nabla_{y}f(r,\Sigma_{r}^{m})D_{\theta}Y_{r}^{m}
-\nabla_{y}f(r,\Sigma_{r})Y_{r}^{\theta}
+\nabla_{z}f(r,\Sigma_{r}^{m})D_{\theta}Z_{r}^{m}
-\nabla_{z}f(r,\Sigma_{r})Z_{r}^{\theta}\Big)dr\\
&+\int_{s}^{T}\!\!\!\Big((\nabla_{x}h(r,\Sigma_{r}^{m})
\!-\!\nabla_{x}h(r,\Sigma_{r}))D_{\theta}X_{r}^{t,x}
\!+\!\nabla_{y}h(r,\Sigma_{r}^{m})D_{\theta}Y_{r}^{m}
\!-\!\nabla_{y}h(r,\Sigma_{r})Y_{r}^{\theta}\Big)\overleftarrow{dB_{r}}\\
&+\int_{s}^{T}\Big(\nabla_{z}h(r,\Sigma_{r}^{m})D_{\theta}Z_{r}^{m}
-\nabla_{z}h(r,\Sigma_{r})Z_{r}^{\theta}\Big)\overleftarrow{dB_{r}}\\
&-\int_{s}^{T}(D_{\theta}Z_{r}^{m+1}-Z_{r}^{\theta})dW_{r}.
\end{align*}
From Proposition \ref{aprioriest}, we have
\begin{align*}
&E[\sup_{\theta \leq s\leq T}|D_{\theta}Y_{s}^{m+1}-Y_{s}^{\theta}|^{2}]+E[\int_{s}^{T}|D_{\theta}Z_{r}^{m+1}-Z_{r}^{\theta}|^{2}dr]\\
&\leq CE\Big[\int_{s}^{T}\Big|(\nabla_{x}f(r,\Sigma_{r}^{m})
-\nabla_{x}f(r,\Sigma_{r}))D_{\theta}X_{r}^{t,x}+\nabla_{y}f(r,\Sigma_{r}^{m})Y_{r}^{\theta}
-\nabla_{y}f(r,\Sigma_{r})Y_{r}^{\theta}\\
&+\nabla_{z}f(r,\Sigma_{r}^{m})Z_{r}^{\theta}
-\nabla_{z}f(r,\Sigma_{r})Z_{r}^{\theta}\Big|^2dr\Big]\\
&+CE\Big[\int_{s}^{T}\Big|(\nabla_{x}h(r,\Sigma_{r}^{m})
-\nabla_{x}h(r,\Sigma_{r}))D_{\theta}X_{r}
+\nabla_{y}h(r,\Sigma_{r}^{m})Y_{r}^{\theta}
-\nabla_{y}h(r,\Sigma_{r})Y_{r}^{\theta}\\
&+\nabla_{z}h(r,\Sigma_{r}^{m})Z_{r}^{\theta}
-\nabla_{z}h(r,\Sigma_{r})Z_{r}^{\theta}\Big|^{2}dr\Big].
\end{align*}
Therefore, we obtain
\begin{align}\label{YmYth}
&E\Big[\int_t^T\int_t^T|D_{\theta}Y_s^{m+1}-Y^{\theta}_s|^{2}dsd\theta
+ \int_t^T\int_t^T|D_{\theta}Z_s^{m+1}-Z^{\theta}_s|^{2}dsd\theta\Big]\nonumber\\
&\leq C E\Big[\int_t^T\int_{t}^{T}|\delta_{r,\theta}^{m}|^2dr d\theta +\int_t^T\int_{t}^{T}|\rho_{r,\theta}^{m}|^2dr d\theta\Big],
\end{align}
where
\begin{align*}
\delta_{r,\theta}^{m}&=(\nabla_{x}f(r,\Sigma_{r}^{m})
-\nabla_{x}f(r,\Sigma_{r}))D_{\theta}X_{r}^{t,x}
+\nabla_{y}f(r,\Sigma_{r}^{m})Y_{r}^{\theta}
-\nabla_{y}f(r,\Sigma_{r})Y_{r}^{\theta}\\
&+\nabla_{z}f(r,\Sigma_{r}^{m})Z_{r}^{\theta}
-\nabla_{z}f(r,\Sigma_{r})Z_{r}^{\theta}
\end{align*}
and
\begin{align*}
\rho_{r,\theta}^{m}&=(\nabla_{x}h(r,\Sigma_{r}^{m})
-\nabla_{x}h(r,\Sigma_{r}))D_{\theta}X_{r}^{t,x}
+\nabla_{y}h(r,\Sigma_{r}^{m})Y_{r}^{\theta}
-\nabla_{y}h(r,\Sigma_{r})Y_{r}^{\theta}\\
&+\nabla_{z}h(r,\Sigma_{r}^{m})Z_{r}^{\theta}
-\nabla_{z}h(r,\Sigma_{r})Z_{r}^{\theta}.
\end{align*}
From the definition of  $(\delta_{r,\theta}^{m})_{t\leq r,\theta\leq T}$, we have
\begin{align*}
E\Big[\int_t^T\int_{t}^{T}|\delta_{r,\theta}^{m}|^2dr d\theta\Big]\leq C
\int_t^T(A_{m}(\theta,t,T)+B_{m}(\theta,t,T))d\theta,
\end{align*} 
 where
\begin{eqnarray*}
A_{m}(\theta,t,T)&=&E\Big[\int_{t}^{T}|(\nabla_{x}f(r,\Sigma_{r}^{m})
-\nabla_{x}f(r,\Sigma_{r}))D_{\theta}X_{r}^{t,x}|^2 dr\Big],\\
B_{m}(\theta,t,T)&=&E\Big[\int_{t}^{T}\big|(\nabla_{y}f(r,\Sigma_{r})
-\nabla_{y}f(r,\Sigma_{r}^{m}))Y_{r}^{\theta}\big|^2dr\Big]\\
&+&E\Big[\int_{t}^{T}\big|(\nabla_{z}f(r,\Sigma_{r})
-\nabla_{z}f(r,\Sigma_{r}^{m}))Z_{r}^{\theta}\big|^2dr\Big].
\end{eqnarray*}
Moreover, since $\nabla_{x}f$ is bounded and continuous with respect to $(x,y,z)$,
it follows by the dominated convergence theorem and inequality \ref{integrability}  that
\begin{equation}\label{A}
\lim_{m\rightarrow \infty} \int_t^TA_{m}(\theta,t,T)d\theta=0.
\end{equation}
Furthermore, since $\nabla_{y}f$ and $\nabla_{z}f$ are bounded and continuous with respect to $(x,y,z)$,
it follows also by the dominated convergence theorem and inequality \ref{apriori1} that
\begin{equation}\label{B}
\lim_{m \rightarrow \infty}\int_t^TB_{m}(\theta,t,T)d\theta=0.
\end{equation}
From the definition of  $(\rho_{r,\theta}^{m})_{s\leq r,\theta\leq T}$, we have
\begin{align*}
E\Big[\int_t^T\int_{t}^{T}|\rho_{r,\theta}^{m}|^2dr d\theta\Big]\leq C \int_t^T(A'_{m}(\theta,t,T)+B'_{m}(\theta,t,T))d\theta,
\end{align*}
with
\begin{eqnarray*}\label{ABp}
A'_{m}(\theta,t,T)&=&E\Big[\int_{t}^{T}|(\nabla_{x}h(r,\Sigma_{r}^{m})
-\nabla_{x}h(r,\Sigma_{r}))D_{\theta}X_{r}^{t,x}|^2 dr\Big],\\
B'_{m}(\theta,t,T)&=&E\Big[\int_{t}^{T}\big|(\nabla_{y}h(r,\Sigma_{r})
-\nabla_{y}h(r,\Sigma_{r}^{m}))Y_{r}^{\theta}\big|^2dr\big]\\
&+&E\Big[\int_{t}^{T}\big|(\nabla_{z}h(r,\Sigma_{r})
-\nabla_{z}h(r,\Sigma_{r}^{m}))Z_{r}^{\theta}\big|^2dr\Big].\nonumber
\end{eqnarray*}
Similarly as shown above, since $\nabla_{y}h$ and $\nabla_{z}h$ are bounded and continuous with respect to $(x,y,z)$ we can
show that
\begin{eqnarray}\label{ABp}
\lim_{m\rightarrow \infty}\int_t^TA'_{m}(\theta,t,T)d\theta
=\lim_{m\rightarrow \infty}\int_t^T B'_{m}(\theta,t,T) d\theta=0.
\end{eqnarray}
Using \ref{A}, \ref{B} and \ref{ABp} in the estimate \ref{YmYth}, we deduce that
\begin{eqnarray*}
\lim_{m\rightarrow \infty}E\Big[\int_t^T\int_t^T|D_{\theta}Y_s^{m+1}-Y^{\theta}_s|^{2}dsd\theta
+ \int_t^T\int_t^T|D_{\theta}Z_s^{m+1}-Z^{\theta}_s|^{2}dsd\theta \Big]=0.
\end{eqnarray*}
It follows that $(Y^{m},Z^{m})$ converges to $(Y,Z)$ in $L^2([t,T],\mathbb{D}^{1,2}\times \mathbb{D}^{1,2})$
and a version of $(D_{\theta}Y,D_{\theta}Z)$ is given by $(Y^{\theta},Z^{\theta})$, which is the desired result.
\end{proof}
\ep\\

\subsection{Representation and path regularity results for the martingale component of the solution of the F-BDSDE}
In this subsection, we prove a representation result for the martingale component $Z$ (that implies a path regularity result) which will be useful to prove the $L^2$-regularity of the solution of the F-BDSDE.
\begin{Proposition}\label{DY version of Z}
Let assumptions {\bf(H1)} and {\bf(H2)} hold and assume that the coefficients $b,\sigma$ and $\Phi$ are in $C_b^{1}$ and $f$ and $h$ are  in $C_b^{0,1}$. Then, 
$\{D^i_sY_s^{t,x},t \leq s \leq T \}$ is a version of $\{ (Z_s^{t,x})_i ,t \leq s \leq T  \}$, where $(Z_s^{t,x})_i$ denotes the $i-th$ component of the matrix $Z_s^{t,x}$.
\end{Proposition}
\begin{proof}
To simplify the notations, we restrict ourselves to the case $k=d=1$.\\
 Notice that for $t\leq s$, we have
\begin{eqnarray*}
Y_{s}^{t,x}=Y_{t}^{t,x}-\int_{t}^{s}f(r,\Sigma_{r}^{t,x})dr - \int_{t}^{s}h(r,\Sigma_{r}^{t,x})
\overleftarrow{dB_{r}}+\int_{t}^{s}Z_{r}^{t,x}dW_{r},
\end{eqnarray*}
where $\Sigma_{r}^{t,x}:=(X_{r}^{t,x},Y_{r}^{t,x},Z_{r}^{t,x})$.\\
It follows from Lemma \ref{forD12} and Lemma \ref{backD12} that, for $t<\theta \leq s$
\begin{align*}
D_{\theta}Y_{s}^{t,x}&=Z_{\theta}^{t,x}
+ \int_{\theta}^{s}D_{\theta}Z_{r}^{t,x}dW_{r}\\
&- \int_{\theta}^{s}\Big( \nabla_{x}f(r,\Sigma_{r}^{t,x})D_{\theta}X_{r}^{t,x}
+ \nabla_{y}f(r,\Sigma_{r}^{t,x})D_{\theta}Y_{r}^{t,x}+\nabla_{z}f(r,\Sigma_{r}^{t,x})
D_{\theta}Z_{r}^{t,x}\Big)dr\\
&-\int_{\theta}^{s}\Big( \nabla_{x}h(r,\Sigma_{r}^{t,x})
D_{\theta}X_{r}^{t,x}+\nabla_{y}h(r,\Sigma_{r}^{t,x})D_{\theta}Y_{r}^{t,x}+\nabla_{z}h(r,\Sigma_{r}^{t,x})
D_{\theta}Z_{r}^{t,x}\Big)\overleftarrow{dB_{r}}.
\end{align*}
The result follows by taking $\theta=s$.
\end{proof}
\ep
\begin{Corollary}
Let assumptions {\bf(H1)} and {\bf(H2)} hold and assume that the coefficients $b,\sigma$ and $\Phi$ are in $C_b^{1}$ and $f$ and $h$ are  in $C_b^{0,1}$. Then, for any $0 \leq t \leq s \leq T$ and $x \in \R^d$,
\begin{equation}\label{Representation of Z}
Z_s^{t,x}= \nabla Y^{t,x}_s [\nabla X_s^{t,x}]^{-1} \sigma ( X_s^{t,x}).
\end{equation}
In particular, $Z^{t,x}$ has continuous paths.
\end{Corollary}
\begin{proof}
Recall that the matrix $\nabla X_s= \Big( \frac{\partial X^i_s }{ \partial x^j}\Big)_{1\leq i,j \leq d}$ solves the SDE
\begin{eqnarray}\label{Grad X}
\nabla X_{s}^{t,x} = I_d + \int_{t}^{s}\nabla b(X_{u}^{t,x})\nabla X_{u}^{t,x}du
+\int_{t}^{s}\nabla \sigma(X_{u}^{t,x})\nabla X_{u}^{t,x}dW_{u}.
\end{eqnarray}
From the uniqueness of the solution of the SDE \eqref{Mall Deriv X SDE} satisfied by $D_{\theta} X_{s}^{t,x}$, it follows that
\begin{equation}\label{DX formula}
D_{\theta} X_{s}^{t,x}= \nabla X_s^{t,x} [\nabla X_{\theta}^{t,x}]^{-1} \sigma ( X_{\theta}^{t,x}).
\end{equation}
Now, consider the equation
\begin{eqnarray}\label{Grad Y}
&&\nabla Y_{s}^{t,x}=\nabla \Phi(X_{T}^{t,x})\nabla X_{T}^{t,x}
+\int_{s}^{T}\Big( \nabla_{x}f(r,X_{r}^{t,x},Y_{r}^{t,x},Z_{r}^{t,x})\nabla X_{r}^{t,x} \nonumber\\
&+&  \nabla_{y}f(r,X_{r}^{t,x},Y_{r}^{t,x},Z_{r}^{t,x})\nabla Y_{r}^{t,x} +  \nabla_{z}f(r,X_{r}^{t,x},Y_{r}^{t,x},Z_{r}^{t,x}) \nabla Z_{r}^{t,x}\Big)dr \nonumber\\
&+& \int_{s}^{T}\Big( \nabla_{x}h(r,X_{r}^{t,x},Y_{r}^{t,x},Z_{r}^{t,x})\nabla X_{r}^{t,x}
+\!\!\! \nabla_{y}h(r,X_{r}^{t,x},Y_{r}^{t,x},Z_{r}^{t,x})\nabla Y_{r}^{t,x} \nonumber\\
&+& \nabla_{z}h(r,X_{r}^{t,x},Y_{r}^{t,x},Z_{r}^{t,x})
\nabla Z_{r}^{t,x} \Big)\overleftarrow{dB_{r}}- \!\int_{s}^{T}\!\!\!\nabla \! Z_{r}^{t,x} dW_{r}.
\end{eqnarray}
Denote by $(\nabla X^{t,x}, \nabla Y^{t,x},\nabla Z^{t,x})$ the solution of the F-BDSDE \eqref{Grad X}-\eqref{Grad Y}. From the uniqueness of the solution of BDSDE \eqref{DY} and the formula \eqref{DX formula}, we deduce that 
\begin{equation}\label{DY formula}
D_{\theta} Y_{s}^{t,x}= \nabla Y_s^{t,x} [\nabla X_{\theta}^{t,x}]^{-1} \sigma ( X_{\theta}^{t,x}).
\end{equation}
Thus
\begin{equation*}\label{DY formula}
D_{s} Y_{s}^{t,x}= \nabla Y_s^{t,x} [\nabla X_s^{t,x}]^{-1} \sigma ( X_s^{t,x}).
\end{equation*}
By Proposition \ref{DY version of Z}, the representation \eqref{Representation of Z} follows. The continuity of $Z^{t,x}$ follows from that of $D_{s} Y_{s}^{t,x}$, which follows from that of $\nabla Y_s^{t,x} $, $\nabla X_s^{t,x} $ and $X_s^{t,x}$.
\end{proof}
\ep

\subsection{Zhang $L^2$-Regularity result under globally Lipschitz continuous assumptions}\label{ssubection Zhang regularity}
In this subsection, we prove the Zhang $L^2$-regularity result for the solution of the F-BDSDE \eqref{forward}-\eqref{1} under globally Lipschitz continuous assumptions on the coefficients. Thus, we extend the results of Zhang \cite{ zhan:04} on F-BSDEs to the doubly stochastic framework.
The following lemma gives estimates and stability results (after a perturbation on the coefficients) for the solution of a F-BDSDE. Its proof is omitted since it is based on technics which are classical in BSDEs' theory.
\begin{Lemma}\label{Lemma stability non markovian }
Assume that assumptions {\bf (H1)} and {\bf(H2)} hold. Let $\displaystyle{(X,Y,Z)}$ denote the solution of the F-BDSDE \eqref{forward}-\eqref{1}.
Then we have the following: \\
(i) $L^p$ estimates: For all $p\geq 2 $, there exists a constant $C_p$ depending only on $T, K, \alpha$ and $p$ such that
\begin{align}\label{Yts Estimation}
E\Big[\Sup_{0 \leq s \leq T}|Y_{s}|^{p} +  \Big(\int_{0}^{T}|Z_{s}|^{2}ds\Big)^{\frac{p}{2}}\Big] &\leq C_p  E\Big[ |\Phi(X_T)|^p +  \int_{0}^{T}|f(s,X_s,0,0)|^p ds\nonumber\\ 
&+ \int_{0}^{T}|h(s,X_s,0,0)|^p ds \Big]
\end{align}
and
\begin{align}\label{Ys-Yt Estimation}
E\Big[|Y_{s}-Y_{t}|^{p} \Big] &\leq C_p \Big\{  E \Big[ |\Phi(X_T)|^p +  \Sup_{0 \leq s \leq T}|f(s,X_s,0,0)|^p  \nonumber\\
&+ \Sup_{0 \leq s \leq T}|h(s,X_s,0,0)|^p ds \Big]|s-t|^{p-1}+  E\Big[\Big(\int_{t}^{s}|Z_{u}|^{2}du\Big)^{\frac{p}{2}}\Big] \Big\}.
\end{align}
(ii)Stability result: Let $(X^{\epsilon},Y^{\epsilon},Z^{\epsilon})$ denote the solution of the perturbed F-BDSDE \eqref{forward}-\eqref{1} with coefficients replaced by $b^\epsilon,\sigma^\epsilon,f^\epsilon$, $h^\epsilon$ and $\Phi^{\epsilon}$ and initial condition replaced by $x^\epsilon$. Assume that $b^\epsilon,\sigma^\epsilon,f^\epsilon$, $h^\epsilon$ and $\Phi^{\epsilon}$ satisfy assumptions {\bf (H1)} and {\bf(H2)}, that $\displaystyle{\lim_{\epsilon \longrightarrow 0} x^\epsilon =x}$ and that for fixed (x,y,z) in $\mathbb{R}^{d}\times\mathbb{R}^{k}\times\mathbb{R}^{k\times d}$
\begin{align*}
&\lim_{\epsilon \longrightarrow 0} |b^\epsilon(x) - b(x)|^{2} + |\sigma^\epsilon(x) - \sigma(x)|^{2} = 0,\\
&\lim_{\epsilon \longrightarrow 0} \! |\Phi^{\epsilon}(x) - \Phi(x)|^{2}\!\! +\!\!\int_{0}^{T}\!\!\!\!\!|h^\epsilon(s,x,y,z) \!-\! h(s,x,y,z))|^{2}ds\\
&\qquad \qquad \qquad \qquad  + \int_{0}^{T}\!\!\!\!|f^\epsilon(s,x,y,z) \!- \!f(s,x,y,z)|^{2}ds =\! 0.
\end{align*}
Then we have
\begin{equation}\label{stability FBDSDE}
\lim_{\epsilon \longrightarrow 0}  E\Big[ \Sup_{0 \leq s \leq T} |X_{s}^{\epsilon} - X_{s}|^{2} + \Sup_{0 \leq s \leq T}  |Y_{s}^{\epsilon} - Y_{s}|^{2} + \int_{0}^{T} |Z_{s}^{\epsilon} - Z_{s}|^{2}ds \Big]=0.
\end{equation}
\end{Lemma}
The next lemma provides $L^p$ estimates for the martingale component $Z$ of the solution of the F-BDSDE \ref{forward}-\ref{1}, for $p\geq 2$. It gives also estimates for the continuous component $Y$.
\begin{Lemma}\label{lemma Z bounded Y regularity}
 Assume that assumptions {\bf (H1)} and {\bf(H2)} hold. Then for all $p \geq 2$, there exists a constant $C_p > 0$ depending only on $T,K, \alpha$ and $p$ such that
\begin{eqnarray}\label{Zs bounded}
\Big(E[|Z_s^{t,x}|^p]\Big)^{\frac{1}{p}} \leq C_p(1+|x|) \textrm{ }a.e.\textrm{ }s \in [t,T].
\end{eqnarray}
In addition, there exists a positive constant $C$ independent from $\hat{h}$ the time step of a given uniform time-grid $\pi:=\{0=t_0<\ldots<t_N=T\}$ such that
\begin{align}\label{Y regularity}
\max_{0\leq n \leq N-1} \sup_{t_n \leq s \leq t_{n+1}}E\Big[|Y_{s}^{t,x}-Y_{t_n}^{t,x}|^{2}] + |Y_{s}^{t,x}-Y_{t_{n+1}}^{t,x}|^{2}\Big]  \leq C \hat{h} (1+|x|^2).
\end{align}
\end{Lemma}

\begin{proof}
Fisrt, we consider the case when $b,\sigma$ and $\Phi$ are in $C_b^{1}$ and $f$ and $h$  are in $C_b^{0,1}$ and satisfying assumptions {\bf (H1)} and {\bf(H2)}.
Let $(\nabla X^{t,x}, \nabla Y^{t,x}, \nabla Z^{t,x})$ be the solution of the F-BDSDE \eqref{Grad X}-\eqref{Grad Y}.\\
Since $\nabla X^{t,x}$ is the solution of the SDE $(\ref{Grad X})$, $[\nabla X^{t,x}]^{-1}$ is also the solution of an SDE and we have the following estimate
\begin{eqnarray}
E\Big[\Sup_{0 \leq t \leq T}|[\nabla X_{s}^{t,x}]^{-1}|^p\Big] \leq C_p.
\end{eqnarray}
On the other hand, $\nabla Y^{t,x}$ is the solution of the linear BDSDE $(\ref{Grad Y})$. Using estimate $(\ref{Yts Estimation})$, we get
\begin{eqnarray}
E\Big[\Sup_{0 \leq t \leq T}|\nabla Y_{s}^{t,x}|^p\Big] \leq C_p.
\end{eqnarray}
Now, recall the representation result \eqref{Representation of Z}
\begin{equation*}
Z_{s}^{t,x} = \nabla Y_{s}^{t,x} [\nabla X_{s}^{t,x}]^{-1}  \sigma(X_{s}^{t,x}), P-a.s. \textrm{, for all }  s \in [t,T].
\end{equation*}
Using H\"{o}lder's inequality, we get
\begin{eqnarray}\label{Estimation Z regular case}
(E[|Z_{s}^{t,x}|^{p}] )^{\frac{1}{p}}&\leq& (E[|\nabla Y_{s}^{t,x}|^{3p}] )^{\frac{1}{3p}}  (E[| [\nabla X_{s}^{t,x}]^{-1}|^{3p}] )^{\frac{1}{3p}} (E[|\sigma(X_{s}^{t,x})|^{3p}] )^{\frac{1}{3p}}\nonumber\\
&\leq& C_p (1+|x|),\textrm{ } \forall s \in [t,T].
\end{eqnarray}
Now the aim is to generalize the previous estimate to the globally Lipschitz continuous coefficients' case. So let $b,\sigma,\Phi,f$ and $h$ be coefficients satisfying the assumptions $(\bf{H1})$ and $(\bf{H2})$ and let $b^k,\sigma^k,\Phi^k,f^k$ and $h^k$ be smooth molifiers of these coefficients (take $b^k,\sigma^k,\Phi^k$ in $C^1_b$ and $f^k$ and $h^k$ in $C^{0,1}_b$). Denoting $Z^{t,x,k}$ the solution of the F-BDSDE associated to the smooth coefficients, we deduce from $(\ref{Estimation Z regular case})$ that $(E[|Z_{s}^{k,t,x}|^{p}] )^{\frac{1}{p}} \leq C_p (1+|x|)$, $\forall s \in [t,T]$, where $C_p$ is independent from $k$. Using the stability result $(\ref{stability FBDSDE})$ , we get
\begin{eqnarray}
\lim_{k \longrightarrow + \infty } E\Big[\int_{t}^{T}|Z_{s}^{k,t,x}-Z_{s}^{t,x}|^2 ds \Big] =0.
\end{eqnarray}
We deduce that for a.e. $s \in [t,T]$, there exist a subsequence of $(Z^{k,t,x})_k$ such that $\displaystyle{\lim_{k \longrightarrow + \infty } Z_{s}^{k,t,x}=Z_{s}^{t,x}}$ in probability. By the Fatou's Lemma, we get $(E[|Z_{s}^{t,x}|^{p}] )^{\frac{1}{p}} \leq C_p (1+|x|)$. Inserting the latter inequality in  $(\ref{Ys-Yt Estimation})$, we get the estimate $(\ref{Y regularity})$.
\end{proof}
\ep\\
Now we are in position to prove our main result which is the $L^2$-regularity of the solutions of F-BDSDEs.
Fisrt, we need to define the step process $\bar{Z}$.\\ Let 
$\displaystyle{\pi:=\{0=t_0<\ldots<t_N=T\}}$ be a uniform time-grid with time step $\hat{h}$. We define $\bar{Z}$ by
\begin{eqnarray}\label{barZdef}
\left\{ \begin{array}{lll}
\bar{Z}_{t}=\frac{1}{\hat{h}}\displaystyle{E_{t_{n}}\Big[\int_{t_{n}}^{t_{n+1}} Z_{s}ds\Big]}, \textrm{ for all } t \in [t_n,t_{n+1}), \textrm{ for all } n \in \{0,\ldots,N-1\},\\
\bar{Z}_{t_{N}}=0.
\end{array} \right.
\end{eqnarray}
\begin{Theorem}[$L^2$-regularity]\label{Theorem L2 regularity}
Under assumptions $(\bf{H1})$ and $(\bf{H2})$, we have
\begin{align}\label{L2 regularity estimate}
\underset{0 \leq n \leq N-1}{\max} & \Sup_{t_{n} \leq s \leq _{n+1}}   E\Big[|Y_{s}-Y_{t_{n}}|^{2} + |Y_{s}-Y_{t_{n+1}}|^{2}\Big] \nonumber\\
&+ \sum_{n=0}^{N-1} \int_{t_{n}}^{t_{n+1}}E\big[ |Z_{s}-\bar{Z}_{t_{n}}|^{2} +  |Z_{s}-\bar{Z}_{t_{n+1}}|^{2} \big]ds
 \leq  C \hat{h} (1+|x|^2) .
\end{align}
\end{Theorem}

\begin{proof}
Using the estimate \eqref{Y regularity}, one obtains 
\begin{align}\label{Upper bound Appendix}
&\underset{0 \leq n \leq N-1}{\max}  \Sup_{t_{n} \leq s \leq _{n+1}}   E\Big[|Y_{s}-Y_{t_{n}}|^{2} + |Y_{s}-Y_{t_{n+1}}|^{2}\Big] \nonumber\\
&+ \sum_{n=0}^{N-1} \int_{t_{n}}^{t_{n+1}}E\big[ |Z_{s}-\bar{Z}_{t_{n}}|^{2} +  |Z_{s}-\bar{Z}_{t_{n+1}}|^{2} \big]ds
 \nonumber\\
 &\leq  C \hat{h} (1+|x|^2) + C \sum_{n=0}^{N-1} \int_{t_{n}}^{t_{n+1}}\!\!\!\!\!  E[|Z_{s}-\bar{Z}_{t_{n}}|^{2}]ds + C \sum_{n=0}^{N-1}\! \int_{t_{n}}^{t_{n+1}}\!\!\!\!\!E[|Z_{s}-\bar{Z}_{t_{n+1}}|^{2}]ds.
\end{align}
Let $b^k,\sigma^k,\Phi^k,f^k$ and $h^k$ be smooth molifiers of $b,\sigma,\Phi,f$ and $h$ (we take $b^k,\sigma^k$ and $\Phi^k$ in $C_b^{1}$ and $f^k$ and $h^k$ in $C_b^{0,1}$). We denote by $(X^k,Y^k,Z^{k})$ the solution of the F-BDSDE associated to the smooth coefficients.\\
First, we deal with the term $ \displaystyle{\sum_{n=0}^{N-1} \int_{t_{n}}^{t_{n+1}}E[|Z_{s}-\bar{Z}_{t_{n}}|^{2}]ds}$.
Since the conditional expectation minimizes the conditional mean square error, we have
\begin{eqnarray*}\label{}
&&\sum_{n=0}^{N-1} \int_{t_{n}}^{t_{n+1}}E|Z_{s}-\bar{Z}_{t_{n}}|^{2}ds  \leq   \sum_{n=0}^{N-1} \int_{t_n}^{t_{n+1}} E|Z_s -Z_{t_n}^{k}|^2ds\nonumber\\
&&\leq  2 \sum_{n=0}^{N-1} \int_{t_{n}}^{t_{n+1}}E|Z_{s}-Z_{s}^{k} |^{2}ds  + 2 \sum_{n=0}^{N-1} \int_{t_n}^{t_{n+1}} E|Z_{s}^{k} -Z_{t_n}^{k} |^2ds, \nonumber\\
&&=  2 \int_{0}^{T} | Z_{s} - Z_{s}^{k}  |^{2}ds  + 2 \sum_{n=0}^{N-1} \int_{t_n}^{t_{n+1}} E|Z_{s}^{k} -Z_{t_n}^{k} |^2ds. \nonumber\\
\end{eqnarray*}
By the stability result (\ref{stability FBDSDE}), we have
\begin{align}\label{stability FBDSDE for Zk}
\lim_{k \longrightarrow +\infty} \int_{0}^{T} |Z_{s}^{k} - Z_{s}|^{2}ds=0.
\end{align}
Now, using the representation result \eqref{Representation of Z} for $Z^k$, we have
\begin{eqnarray}\label{Zks-Zks' representation}
Z_{s}^{k} -Z_{s'}^{k} = \nabla Y_{s}^{k} [\nabla X_{s}^{k}]^{-1}  \sigma^{k}(X_{s}^{k})- \nabla Y_{s'}^{k} [\nabla X_{s'}^{k}]^{-1}  \sigma^{k}(X_{s'}^{k}), s, s' \in [t_n,t_{n+1}).
\end{eqnarray}
Then, by inserting $\nabla Y_{s'}^{k}  [\nabla X_{s}^{k}]^{-1} \sigma^{k}(X_{s}^{k})$ and $\nabla Y_{s'}^{k} [\nabla X_{s'}^{k}]^{-1} \sigma^{k}(X_{s}^{k})$, we obtain
\begin{eqnarray*}
|Z_{s}^{k} -Z_{s'}^{k}|^2 &\leq & 3|\nabla Y_{s}^{k} - \nabla Y_{s'}^{k}|^2 | [\nabla X_{s}^{k}]^{-1}|^2 |\sigma^{k}(X_{s}^{k})|^2\nonumber\\
&+&  3|\nabla Y_{s'}^{k} |^2 | [\nabla X_{s}^{k}]^{-1} -  [\nabla X_{s'}^{k}]^{-1}|^2 |\sigma^{k}(X_{s}^{k})|^2 \nonumber\\
&+&  3|\nabla Y_{s'}^{k} |^2 | [\nabla X_{s'}^{k}]^{-1}|^2 |\sigma^{k}(X_{s}^{k}) - \sigma^{k}(X_{s'}^{k})|^{2}.
\end{eqnarray*}
 For $s'=t_n$, we get
\begin{eqnarray*}
|Z_{s}^{k} -Z_{t_{n}}^{k}|^2 &\leq& C\Big\{|\nabla Y_{s}^{k} - \nabla Y_{t_{n}}^{k}|^2 | [\nabla X_{s}^{k}]^{-1}|^2 |\sigma^{k}(X_{s}^{k})|^2\nonumber\\
&+&  |\nabla Y_{t_{n}}^{k} |^2 | [\nabla X_{s}^{k}]^{-1} -  [\nabla X_{t_{n}}^{k}]^{-1}|^2  |\sigma^{k}(X_{s}^{k})|^2 \nonumber\\
&+&  |\nabla Y_{t_{n}}^{k} |^2 | [\nabla X_{t_{n}}^{k}]^{-1}|^2 |\sigma^{k}(X_{s}^{k}) - \sigma^{k}(X_{t_{n}}^{k})|^{2}\Big\}.
\end{eqnarray*}
We conclude by using H\"{o}lder's inequality and the estimate \eqref{Y regularity} that
\begin{equation}\label{Zktn estimate}
\Sum_{n=0}^{N-1} E\Big[ \int_{t_n}^{t_{n+1}} |Z_s^{k} -Z_{t_n}^{k}|^2ds \Big] \leq C \hat{h} (1+|x|^2),
\end{equation}
here we also used also the same kind of estimation as $(\ref{Ys-Yt Estimation})$ but for $[\nabla X^{k}]^{-1}$  (instead of $\nabla Y_{s}^{k}$) as it is a solution of an SDE.\\
Now, it reminds to handle the error term $$ \sum_{n=0}^{N-1}\! \int_{t_{n}}^{t_{n+1}} E[|Z_{s}-\bar{Z}_{t_{n+1}}|^{2}]ds=\sum_{n=0}^{N-2}\! \int_{t_{n}}^{t_{n+1}} E[|Z_{s}-\bar{Z}_{t_{n+1}}|^{2}]ds + \int_{t_{N-1}}^{t_{N}} E[|Z_{s}|^{2}]ds .$$
Define
\begin{eqnarray}\label{barZdef}
\left\{ \begin{array}{lll}
\bar{Z}^k_{t}=\frac{1}{\hat{h}}\displaystyle{E_{t_{n}}\Big[\int_{t_{n}}^{t_{n+1}} Z_{s}^k ds \Big]}, \textrm{ for all } t \in [t_n,t_{n+1}), \textrm{ for all } n \in \{0,\ldots,N-1\},\\
\bar{Z}^k_{t_{N}}=0.
\end{array} \right.
\end{eqnarray}
Then, inserting $Z_s^{k},Z^k_{t_{n+1}}$ and $\bar{Z}^k_{t_{n+1}}$, we get
\begin{eqnarray*}
&& \sum_{n=0}^{N-2}  \int_{t_{n}}^{t_{n+1}}  E[|Z_{s}-\bar{Z}_{t_{n+1}}|^{2}]ds  \leq  C \sum_{n=0}^{N-2}\! \int_{t_{n}}^{t_{n+1}} E[|Z_{s}- Z_{s}^k|^{2}]ds \nonumber\\ 
&&+ C\sum_{n=0}^{N-2}\! \int_{t_{n}}^{t_{n+1}} E[|Z_{s}^k- Z^k_{t_{n+1}}|^{2}]ds 
+ C\sum_{n=0}^{N-2}\! \int_{t_{n}}^{t_{n+1}} E[|Z^k_{t_{n+1}}-\bar{Z}^k_{t_{n+1}}|^{2}]ds
\nonumber\\
&&+ C\sum_{n=0}^{N-2}\! \int_{t_{n}}^{t_{n+1}} E[|\bar{Z}^k_{t_{n+1}}-\bar{Z}_{t_{n+1}}|^{2}]ds.
\end{eqnarray*}
Note that $$\sum_{n=0}^{N-2}\! \int_{t_{n}}^{t_{n+1}} E[|Z_{s}- Z_{s}^k|^{2}]ds \leq \sum_{n=0}^{N-1}\! \int_{t_{n}}^{t_{n+1}} E[|Z_{s}- Z_{s}^k|^{2}]ds= \int_{0}^{T} | Z_{s} - Z_{s}^{k} |^{2}ds,$$
which tends to zero when $k$ tends to infinity, again by the stability result \ref{stability FBDSDE for Zk}.\\
The term $\displaystyle{\sum_{n=0}^{N-2}\! \int_{t_{n}}^{t_{n+1}} E[|Z_{s}^k- Z^k_{t_{n+1}}|^{2}]ds}$ is bounded by $\displaystyle{\sum_{n=0}^{N-1}\! \int_{t_{n}}^{t_{n+1}} E[|Z_{s}^k- Z^k_{t_{n+1}}|^{2}]ds}$ which is handled exactly like $\displaystyle{ \sum_{n=0}^{N-1}\! \int_{t_{n}}^{t_{n+1}} E[|Z_{s}^k- Z^k_{t_{n}}|^{2}]ds }$ using the representation result \ref{Representation of Z} for $Z^k$ (take $s'=t_{n+1}$ in \eqref{Zks-Zks' representation}). We get
\begin{equation*}
\Sum_{n=0}^{N-1} E\Big[ \int_{t_n}^{t_{n+1}} |Z_s^{k} -Z_{t_{n+1}}^{k}|^2ds \Big] \leq C \hat{h} (1+|x|^2).
\end{equation*}
We deal with the term $\displaystyle{ \sum_{n=0}^{N-2}\! \int_{t_{n}}^{t_{n+1}} E[|Z^k_{t_{n+1}}-\bar{Z}^k_{t_{n+1}}|^{2}]ds} $ as follows.\\
By the definition of $\bar{Z}^k_{t_{n+1}}$, Jensen's inequality and Cauchy-Schwarz's inequality, we have for all $\displaystyle{n=0,\ldots,N-2}$
\begin{eqnarray*}
\int_{t_{n}}^{t_{n+1}} E[|Z^k_{t_{n+1}}-\bar{Z}^k_{t_{n+1}}|^{2}]ds &=&\hat{h} E[|Z_{t_{n+1}}^k-\bar{Z}^k_{t_{n+1}}|^{2}] \\
&=& \frac{1}{\hat{h}}E\Big[\Big|E_{t_{n+1}}\Big[ \int_{t_{n+1}}^{t_{n+2}} (Z^k_{t_{n+1}}-Z_{s}^k ) ds\Big]\Big|^{2}\Big]  \nonumber\\
&\leq &    \frac{1}{\hat{h}}E\Big[ \Big|  \int_{t_{n+1}}^{t_{n+2}} (Z^k_{t_{n+1}} - Z^k_{s} ) ds \Big|^{2}\Big]  \nonumber\\
&\leq &     \int_{t_{n+1}}^{t_{n+2}} E |Z^k_{t_{n+1}}-Z^k_{s}  |^{2} ds. \nonumber\\
\end{eqnarray*}
Thus,
\begin{eqnarray*}
\sum_{n=0}^{N-2} \int_{t_{n}}^{t_{n+1}} E[|Z^k_{t_{n+1}}-\bar{Z}^k_{t_{n+1}}|^{2}]ds & \leq &  \sum_{n=0}^{N-2} \int_{t_{n+1}}^{t_{n+2}} E |Z^k_{t_{n+1}}-Z^k_{s}  |^{2} ds \nonumber\\
 &=&  \sum_{n=1}^{N-1} \int_{t_{n}}^{t_{n+1}} E |Z^k_{t_{n}}-Z^k_{s}  |^{2} ds\nonumber\\
 &\leq &  \sum_{n=0}^{N-1} \int_{t_{n}}^{t_{n+1}} E |Z^k_{t_{n}}-Z^k_{s}  |^{2} ds \nonumber\\
 &\leq & C \hat{h} (1+|x|^2),
\end{eqnarray*}
by \eqref{Zktn estimate}.\\
Finally, we deal with $\displaystyle{\sum_{n=0}^{N-2} \! \int_{t_{n}}^{t_{n+1}} E[|\bar{Z}^k_{t_{n+1}}-\bar{Z}_{t_{n+1}}|^{2}]ds}$ as follows.
By the definitions of $\bar{Z}^k_{t_{n+1}}$ and $\bar{Z}_{t_{n+1}}$, Jensen's inequality and Cauchy-Schwarz's inequality, we have for all $\displaystyle{n=0,\ldots,N-2}$
\begin{eqnarray*}
\int_{t_{n}}^{t_{n+1}} E[|\bar{Z}^k_{t_{n+1}}-\bar{Z}_{t_{n+1}}|^{2}]ds &=&\hat{h} E[|\bar{Z}^k_{t_{n+1}}-\bar{Z}_{t_{n+1}} |^{2}] \\
&=& \frac{1}{\hat{h}}E\Big[\Big|E_{t_{n+1}}\Big[ \int_{t_{n+1}}^{t_{n+2}} (Z^k_{s}-Z_{s} ) ds\Big]\Big|^{2}\Big]  \nonumber\\
&\leq &    \frac{1}{\hat{h}}E\Big[ \Big|  \int_{t_{n+1}}^{t_{n+2}} (Z^k_{s}-Z_{s} ) ds \Big|^{2}\Big]  \nonumber\\
&\leq &     \int_{t_{n+1}}^{t_{n+2}} E | Z^k_{s}-Z_{s}   |^{2} ds. \nonumber\\
\end{eqnarray*}
Hence
\begin{eqnarray*}
\sum_{n=0}^{N-2} \int_{t_{n}}^{t_{n+1}} E[|\bar{Z}^k_{t_{n+1}}-\bar{Z}_{t_{n+1}}|^{2}]ds & \leq &  \sum_{n=0}^{N-2} \int_{t_{n+1}}^{t_{n+2}} E |Z^k_{s}-Z_{s}  |^{2} ds \nonumber\\
 &=&  \sum_{n=1}^{N-1} \int_{t_{n}}^{t_{n+1}} E |Z^k_{s}-Z_{s}    |^{2} ds\nonumber\\
 &\leq &  \sum_{n=0}^{N-1} \int_{t_{n}}^{t_{n+1}} E |Z^k_{s}-Z_{s} |^{2} ds \nonumber\\
 &= & \int_{0}^{T} | Z_{s}^{k}  -  Z_{s}  |^{2}ds,
\end{eqnarray*}
which tends to zero when $k$ goes to infinity by \eqref{stability FBDSDE for Zk}.
To conclude the proof, observe that by \eqref{Zs bounded}, $\displaystyle{\int_{t_{N-1}}^{t_{N}} E[|Z_{s}|^{2}]ds \leq C \hat{h} (1+|x|^2)}$.
\end{proof}
\ep

\section{Application: Rate of convergence in time for a numerical scheme for F-BDSDEs under globally Lipschitz continuous conditions}
In this section, we give the main application of our $L^2$-regularity result stated in Theorem \ref{Theorem L2 regularity}. This application will be in Corollary \ref{Corollary Vitesse de convergence Lipschitz} where we derive, under globally Lipschitz continuous conditions, a rate of convergence in time for the numerical scheme for the F-BDSDE \eqref{forward}-\eqref{1} studied in \cite{bach:benl:mato:mnif:16}.
\subsection{Numerical scheme for F-BDSDEs}\label{Num scheme BDSDEs}
We recall from \cite{bach:benl:mato:mnif:16}
the following discretized version of \eqref{forward}-\eqref{1}. Let
\begin{eqnarray*}\label{disctime}
\pi:t_{0}=0<t_{1}
<\ldots<t_{N}=T,
\end{eqnarray*}
be a partition of the time interval $[0,T]$. For simplicity,
 we take an equidistant partition of  $[0,T]$ i.e. $\hat{h} = \frac{T}{N}$ and
$t_n=n \hat{h}$, $0\leq n\leq N$.
 In the sequel, the notations $\Delta W_{n}=W_{t_{n+1}}-W_{t_{n}}$ and $\Delta B_{n}=B_{t_{n+1}}-B_{t_{n}}$, for $n=0,\ldots,N-1$ will be used.\\
The forward component $X$ is approximated by the classical forward Euler scheme:
 \begin{eqnarray}\label{Xn}
\left\{\begin{array}{ll}
X_{t_{0}}^{N}&=x,\\
 X_{t_{n+1}}^{N}&=X_{t_{n}}^{N}+\hat{h} b(X_{t_{n}}^{N})+\sigma(X_{t_{n}}^{N})\Delta W_{n}, \textrm{ for } n=0,\ldots,N-1.
\end{array}
\right.
\end{eqnarray}
The solution $(Y,Z)$ of (\ref{1}) is approximated by $(Y^{N},Z^{N})$ defined by
\begin{equation*}\label{4}
Y_{t_{N}}^{N} = \Phi(X_{T}^{N})\textrm{ and } Z_{t_{N}}^{N}=0,
\end{equation*}
and for $n=N-1,\ldots,0$, we set
\begin{equation}\label{Yn}
Y_{t_{n}}^{N} = E_{t_{n}}\Big[Y_{t_{n+1}}^{N}+h(t_{n+1},\Theta_{n+1}^{N})\Delta B_{n} + \hat{h} f(t_{n},\Theta_{n}^{N})\Big],
\end{equation}
\begin{equation*}\label{Zn}
\hat{h} Z_{t_{n}}^{N} = E_{t_{n}}\Big[Y_{t_{n+1}}^{N}\Delta W_{n}^\top
+ h(t_{n+1},\Theta_{n+1}^{N})
\Delta B_{n}\Delta W^\top_{n}\Big],
\end{equation*}
where
\begin{eqnarray*}
\Theta_{n}^{N}:=(X_{t_{n}}^{N},Y_{t_{n}}^{N},Z_{t_{n}}^{N}), \textrm{for all } n=0,\ldots,N.
\end{eqnarray*}
$^\top$ denotes the transpose operator and
$E_{t_{n}}$ denotes the conditional expectation w.r.t. the $\sigma$-algebra $\Fc_{t_{n}}$.\\
We also recall the continuous approximation of the solution of BDSDE (\ref{1}). For $\displaystyle{n=0,\ldots,N-1}$
\begin{equation}\label{1N}
Y_t^N:=Y_{t_{n+1}}^{N} +  \!\displaystyle{\int_{t}^{t_{n+1}}\!\!\!\!\!\!\!\!f(t_{n},\Theta_{n}^{N})ds}
+\!\displaystyle{\int_{t}^{t_{n+1}}\!\!\!\!\!\!\!\!h(t_{n+1},\Theta_{n+1}^{N})\overleftarrow{dB_{s}}}
-\!\displaystyle{\int_{t}^{t_{n+1}}\!\!\!\!\!\!Z_{s}^{N} dW_s},\textrm{ } t_{n}\leq t < t_{n+1}.
\end{equation}

\subsection{Rate of convergence for the Euler time discretization based numerical scheme for F-BDSDEs}\label{section-discrete-time-error}

In order to derive the rate of convergence in time of the numerical scheme \eqref{Xn}-\eqref{Yn}, the authors  in \cite{bach:benl:mato:mnif:16} proved the $L^2$-regularity for the martingale integrand $Z$ under strong assumptions on the coefficients. Indeed, they assume that the coefficients $b,\sigma$ and $\Phi$ are in $C_b^{2}$ and $f$ and $h$  are in $C_b^{2,2}$. Our $L^2$-regularity result stated in Theorem \ref{Theorem L2 regularity} requires the coefficients to be only globally Lipschitz continuous but enables us to derive the same rate of convergence in time derived in \cite{bach:benl:mato:mnif:16}. This is an important improvement for that numerical scheme.\\
Let us recall the following upper bound result (Theorem 3.1 in \cite{bach:benl:mato:mnif:16}) for the time discretization error. 
\begin{Theorem}[\cite{bach:benl:mato:mnif:16}]\label{theorem upper-bound}
Define the time discretization error by
\begin{eqnarray}\label{error(Y,Z)}
Error_{N}(Y,Z):=\sup_{0\leq s \leq T}E[|Y_{s}-Y_{s}^{N}|^{2}] +\sum_{n=0}^{N-1}\int_{t_n}^{t_{n+1}}E[|Z_{s}-Z_{s}^{N}|^{2}]ds,
\end{eqnarray}
where $Y^{N}$ and $Z^{N}$ are given by (\ref{1N}).
Under assumptions \textbf{(H1)} and \textbf{(H2)} we have
\begin{eqnarray}\label{convergence}
Error_{N}(Y,Z) &\leq&  C \hat{h}(1+|x|^2) + C \sum_{n=0}^{N-1} \int_{t_{n}}^{t_{n+1}}E[|Z_{s}-\bar{Z}_{t_{n}}|^{2}]ds \nonumber\\
&+&  C \sum_{n=0}^{N-1}\! \int_{t_{n}}^{t_{n+1}}\!\!\!\!\!\!\!\!\!\! E[|Z_{s}-\bar{Z}_{t_{n+1}}|^{2}]ds
+ C \sum_{n=0}^{N-1} \int_{t_{n}}^{t_{n+1}}E[|Y_{s}-Y_{t_{n}}|^{2}]ds\nonumber\\
&+& C \sum_{n=0}^{N-1} \int_{t_{n}}^{t_{n+1}}E[|Y_{s}-Y_{t_{n+1}}|^{2}]ds.
\end{eqnarray}
\end{Theorem}
The rate of convergence in time of our scheme under globally Lipschitz continuous assumptions is derived in the next corollary. 
\begin{Corollary}\label{Corollary Vitesse de convergence Lipschitz}
Under Assumptions $(\bf{H1})$ and $(\bf{H2})$, we have
\begin{eqnarray}\label{vitesse de convergence Lipschitz}
Error_{N}(Y,Z)
 &\leq&  C \hat{h} (1+|x|^2) .
\end{eqnarray}
\end{Corollary}

\begin{proof}
The result follows by using the estimate \eqref{L2 regularity estimate} in the upper bound estimate \eqref{convergence}.
\end{proof}
\ep

%

\end{document}